\documentclass{amsart}

\usepackage{amssymb,stmaryrd,mathrsfs,dsfont,amsmath}

\theoremstyle{plain}
\newtheorem{theorem}{Theorem}[section]
\newtheorem{lemma}[theorem]{Lemma}
\newtheorem{proposition}[theorem]{Proposition}
\newtheorem{corollary}[theorem]{Corollary}

\theoremstyle{definition}

\newtheorem{example}[theorem]{Example}

\theoremstyle{remark}

\input xy
\xyoption{all}

\begin{document}

\title[The Laplacian spectrum of power graphs of some finite abelian $\mbox{p}$-groups]{The Laplacian spectrum of power graphs of some finite abelian $\mbox{p}$-groups}

\author[Ankit Raj]{Ankit Raj}
\address{Department of Mathematics, Central University of South Bihar, Patna, India}
\email{ankitraj@cusb.ac.in}

\author[Shubh N. Singh]{Shubh N. Singh}
\address{Department of Mathematics, Central University of South Bihar, Patna, India}
\email{shubh@cub.ac.in}

%\date{...}

\begin{abstract}
The power graph $\mathcal{G}(G)$ of a group $G$ is a simple graph whose vertices are the elements of $G$ and two distinct vertices are adjacent if one is a power of other. In this paper, we investigate the Laplacian spectrum of the power graph $\mathcal{G}(\mathbb{Z}_{p^m}^n)$ of finite abelian $p$-group $\mathbb{Z}_{p^m}^n$. In particular, we prove that the spectrum of group $\mathbb{Z}_{p^m}^n$ is contained in the Laplacian spectrum of graph $\mathcal{G}(\mathbb{Z}_{p^m}^n)$. For a finite abelian group $G$ whose power graph $\mathcal{G}(G)$ is planar, we also prove that the spectrum of group $G$ is contained in the Laplacian spectrum of graph $\mathcal{G}(G)$.
\end{abstract}



\maketitle

\section{Introduction}

The concept of groups and graphs have various points of contact and their connections have been well studied in the literature. Kelarev and Quinn in their seminal paper \cite{kela00} introduced the notion of \emph{power digraph} of a semigroup as a digraph whose vertex set is the semigroup, and there is an arc from vertex $u$ to the other vertex $v$ whenever $v$ is a power of $u$. Motivated by this concept, Chakrabarty, Ghosh and Sen \cite{chak09} defined the \emph{power graph} $\mathcal{G}(S)$ of a semigroup $S$ as a simple graph with $S$ as its vertex set, and there is an edge between two distinct vertices if one is a power of the other. The topic of power graphs has continued to attract the attention of many researchers \cite{aba13}.

\vspace{0.1cm}
It can be easily seen that the power graph of a finite group is connected. For a finite group $G$, Chakrabarty et al. \cite{chak09} proved that the power graph $\mathcal{G}(G)$ is complete if and only if $G$ is a cyclic group of order $p^m$ for some prime number $p$ and non-negative integer $m$. Cameron and Ghosh \cite{came11} proved that two finite abelian groups with isomorphic power graphs are isomorphic. Further, they conjectured that two finite groups with isomorphic power graphs have the same number of elements of each order. In \cite{came10}, Cameron proved the conjecture affirmatively. Cameron and Ghosh \cite{came11} listed finite groups which have the same automorphism group as its power graph.  Chelvam and Sattanathan \cite{chel13} listed  finite abelian groups whose power graphs are planar. The power graphs have been characterized for different types of finite groups \cite{chapan14, chel13, mck14}.

\vspace{0.1cm}
A finite simple graph can be represented by different kinds of square matrices. The eigenvalues of these matrices have been of deep interest for combinatorics and graph theory. In this context, researchers have studied adjacency spectrum and Laplacian spectrum of power graphs of finite groups. Chattopadhyay and Panigrahi \cite{chapan15} studied Laplacian spectrum and algebraic connectivity of the power graph of additive finite cyclic group and the dihedral group. Further, they concluded that the power graph of additive cyclic group of order $n$ and the dihedral group of order $2n$ are Laplacian integral when $n$ is either a prime power or a product of two distinct primes. Mehranian et al. \cite{mehr17} computed the adjacency spectrum of power graphs of cyclic groups, dihedral groups, elementary abelian groups of prime power order and the Mathieu group $M_{11}$. The matrices of the power graphs of certain finite groups have also been studied in different contexts \cite{chat17, chatt17}.

\vspace{0.1cm}
The remainder of the paper is organized as follows. In Section 2, we introduce our notation, recall the necessary concepts, and then state some necessary known results. The main results are presented in Section 3. Finally, Section 4 concludes the paper.

\section{Preliminaries and Notation}
In this section we introduce some basic concepts of groups, graphs, and matrices and fix the notation used throughout this paper. The symbols $p$ and $n$ are always used to denote a prime number, and a positive integer, respectively. We use $\phi(n)$ to denote the value of the Euler's totient function at positive integer $n$.

\vspace{0.1cm}
Let $G$ be a finite group. The order of $G$ is denoted by $|G|$ and the order of an element $g\in G$ by $|g|$. The \emph{spectrum} of $G$, denoted as $\omega(G)$, is the set of its element orders. The cyclic subgroup of $G$ generated by $g\in G$ is denoted by $\langle g\rangle$. The symbol $\mathbb{Z}_m$ stands for the cyclic group of order $m$. We always assume $\mathbb{Z}_m = \{0,1,\ldots, m-1\}$. The notation $\mathbb{Z}_m^n$ means that the direct product of $n$ copies of $\mathbb{Z}_m$. All further unexplained notation and terminology of groups we refer \cite{dum99}.

\vspace{0.1cm}
Let $\Gamma = (V, E)$ be a finite simple graph. The degree of a vertex $v$ of $\Gamma$ is denoted by $\deg(v)$. A complete graph on $n$ vertices is denoted by $K_n$. Let $\Gamma = (V, E)$ and $\Gamma' = (V', E')$ be two disjoint  finite simple graphs. The \emph{union} of $\Gamma$ and $\Gamma'$, denoted $\Gamma \cup \Gamma'$, is a graph whose vertex set is $V \cup V'$ and edge set is $E\cup E'$. We shall write $k\Gamma$ for the union of $k$ disjoint copies of $\Gamma$.  The \emph{join} of $\Gamma$ and $\Gamma'$, denoted $\Gamma + \Gamma'$, is a graph whose vertex set is $V\cup V'$ and edge set is $E\cup E' \cup \{\{v, v'\}\;|\; v\in V,\; v'\in V'\}$. We say that the graph $\Gamma$ is \emph{planar} if $\Gamma$ can be embedded in the plane so that no two edges intersect except at a vertex.

\vspace{0.1cm}
Let $\Gamma$ be a finite simple graph. Its \emph{Laplacian matrix} is the matrix $L(\Gamma) = D(\Gamma) - A( \Gamma)$, where  $D(\Gamma)$ is the diagonal matrix of vertex degrees of $\Gamma$ and $A(\Gamma)$ is the adjacency matrix of $\Gamma$. The \emph{Laplacian polynomial} of $\Gamma$, denoted as $\Theta(\Gamma, x)$, is the characteristic polynomial of $L(\Gamma)$. The eigenvalues of $L(\Gamma)$ are called the \emph{Laplacian eigenvalues} of $\Gamma$. A graph is called \emph{Laplacian integral} if all its Laplacian eigenvalues are integers. The \emph{Laplacian spectrum} of $\Gamma$, denoted as $\mbox{L-spec}(\Gamma)$, is the multiset of its Laplacian eigenvalues. Note that $L(\Gamma)$ is a positive semi-definite matrix. It therefore has non-negative real eigenvalues. If $0 = \mu_1 < \mu_2 < \cdots < \mu_t$ are distinct Laplacian eigenvalues of $\Gamma$ with algebraic multiplicity $m_1, m_2, \ldots, m_t$, respectively, then we shall denote the Laplacian spectrum of $\Gamma$ by $\mbox{L-spec}(\Gamma)= \{\mu_1^{m_1}, \mu_2^{m_2}, \ldots, \mu_t^{m_t}\}$. For further basic definitions concerning graphs and matrices associated with graphs we refer \cite{bapat14, west00}. We end this section with the following known results which we require in the next section.

\begin{theorem}\cite{mohar91}\label{d-union}
Let $\Gamma$ be the disjoint union of graphs $\Gamma_1, \Gamma_2, \ldots, \Gamma_k$. Then
\[\Theta(\Gamma, x) = \prod_{i=1}^{k}\Theta(\Gamma_i, x).\]
\end{theorem}

\begin{theorem}\cite{kel65}\label{d-join}
Let $\Gamma_1$ and $\Gamma_2$ be two disjoint graphs with $n_1$ and $n_2$ vertices, respectively. Then
\[\Theta(\Gamma_1 + \Gamma_2, x) = \frac{x(x-n_1-n_2)}{(x-n_1)(x-n_2)}\Theta(\Gamma_1, x-n_2)\Theta(\Gamma, x-n_1).\]
\end{theorem}

\begin{theorem}\cite{chak09}\label{comp}
Let $G$ be a finite group. The power graph $\mathcal{G}(G)$ is complete if and only if $G$ is a cyclic group of order $p^m$ for some prime number $p$ and non-negative integer $m$.
\end{theorem}

\begin{theorem}\cite{mck14}\label{struct}
For any prime number $p$ and two positive integers $m$ and $n$,
\[\mathcal{G}(\mathbb{Z}_{p^m}^n) \cong K_1 + \frac{p^n-1}{p-1}\left(K_{\phi(p)}+p^{n-1}\left(K_{\phi(p^2)}+p^{n-1}(\cdots + p^{n-1}K_{\phi(p^m)})\right)\cdots\right).\]
\end{theorem}

\section{Main Results}

In \cite{bapt14}, a flower graph is a block graph with only one cut vertex. We know that every non-trivial cyclic group has at least one generator. It is easy to see that the power graph $\mathcal{G}(\mathbb{Z}_m)$  of the cyclic group $\mathbb{Z}_m$ is not a flower graph. By Theorem \ref{struct}, we immediately observe the following.

\begin{theorem}
The power graph $\mathcal{G}(\mathbb{Z}_{p^m}^n)$ is a flower graph if and only if $n\ge 2$ and $m = 1$.
\end{theorem}

One of the original motives of this paper is to determine the Laplacian spectrum of planar power graphs of finite abelian groups. In \cite{chel13}, it is proved that the power graph $\mathcal{G}(G)$ of a finite abelian group $G$ is planar if and only if $G$ is isomorphic to one of the following abelian groups: $\mathbb{Z}_2^n$, $\mathbb{Z}_3^n$, $\mathbb{Z}_4^n$, $\mathbb{Z}_2^r \times \mathbb{Z}_4^s$.  In this context we find the Laplacian spectrum of the power graphs $\mathcal{G}(\mathbb{Z}_{p^m}^n)$ and $\mathcal{G}(\mathbb{Z}_2^r \times \mathbb{Z}_{4}^s)$ in the following two subsections.

\subsection{Laplacian spectrum of power graph $\mathcal{G}(\mathbb{Z}_{p^m}^n)$}
In this subsection we determine the Laplacian polynomial of the power graph $\mathcal{G}(\mathbb{Z}_{p^m}^n)$ and observe that the spectrum of group $\mathbb{Z}_{p^m}^n$ is contained in the Laplacian spectrum of graph  $\mathcal{G}(\mathbb{Z}_{p^m}^n)$. By using  Theorem \ref{d-union} and Theorem \ref{d-join}, we have the following.

\begin{lemma}\label{1-phi}
For each integer $i\;(1\le i \le m-1)$, the Laplacian polynomial $\Theta \left(K_{\phi(p^i)}+p^{n-1}\left(K_{\phi(p^{i+1})}+p^{n-1}(\cdots + p^{n-1}K_{\phi(p^m)})\right)\cdots\right), x-p^{i-1})$ is
\[ (x-p^{i-1})\left(x-p^i-p^{n-1}\left(\phi(p^{i+1})+p^{n-1}(\cdots + p^{n-1}\phi(p^m))\cdots\right)\right)^{\phi(p^i)}\]
\[\frac{\{\Theta \left(K_{\phi(p^{i+1})}+p^{n-1}\left(K_{\phi(p^{i+2})}+p^{n-1}(\cdots + p^{n-1}K_{\phi(p^m)})\right)\cdots\right), x-p^i)\}^{p^{n-1}}}{(x-p^i)}.\]
\end{lemma}

The next theorem gives the Laplacian polynomial of the power graph $\mathcal{G}(\mathbb{Z}_{p^m}^n)$.

\begin{theorem} \label{1-main}
The Laplacian polynomial of the power graph $\mathcal{G}(\mathbb{Z}_{p^m}^n)$ is
\begin{equation*}
\begin{split}
&\quad x(x-p^{mn})(x-1)^{l-1} (x-p^m)^{lp^{(m-1)(n-1)}(\phi(p^m)-1)} \left(\prod_{i=1}^{m-1}(x-p^i)^{lp^{(i-1)(n-1)}(p^{n-1}-1)}\right) \\
&\quad \left(\prod_{i=1}^{m-1}\left(x-p^i-p^{n-1}\left(\phi(p^{i+1})+p^{n-1}(\cdots + p^{n-1}\phi(p^m))\cdots\right)\right)^{lp^{(i-1)(n-1)}\phi(p^i)}\right),
\end{split}
\end{equation*}
where $l = \frac{p^n -1}{p-1}$.
\end{theorem}

\begin{proof}
By using Theorem \ref{d-union} and Theorem \ref{d-join}, we get the Laplacian polynomial

\begin{equation*}
\begin{split}
\Theta(\mathcal{G}(\mathbb{Z}_{p^m}^n), x) & = x(x-p^{mn})\\
&\quad \frac{\{\Theta \left(K_{\phi(p)}+p^{n-1}\left(K_{\phi(p^{2})}+p^{n-1}(\cdots + p^{n-1}K_{\phi(p^m)})\right)\cdots\right), x-1)\}^l}{(x-1)}
\end{split}
\end{equation*}
Using Theorem \ref{d-union}, Theorem \ref{d-join}, and Lemma \ref{1-phi}, the above polynomial will be

\begin{equation*}
\begin{split}
&=x(x-p^{mn})(x-1)^{l-1}\left(x-p-p^{n-1}\left(\phi(p^{2})+p^{n-1}(\cdots + p^{n-1}\phi(p^m))\cdots\right)\right)^{l\phi(p)}\\
&\quad \frac{\{\Theta \left(K_{\phi(p^{2})}+p^{n-1}\left(K_{\phi(p^{3})}+p^{n-1}(\cdots + p^{n-1}K_{\phi(p^m)})\right)\cdots\right), x-p)\}^{lp^{n-1}}}{(x-p)^l}
\end{split}
\end{equation*}

Using Theorem \ref{d-union}, Theorem \ref{d-join}, and Lemma \ref{1-phi}, the above polynomial will be

\begin{equation*}
\begin{split}
&=x(x-p^{mn})(x-1)^{l-1}\left(x-p-p^{n-1}\left(\phi(p^{2})+p^{n-1}(\cdots + p^{n-1}\phi(p^m))\cdots\right)\right)^{l\phi(p)}\\
&\quad (x-p)^{l(p^{n-1}-1)} \left(x-p^2-p^{n-1}\left(\phi(p^{3})+p^{n-1}(\cdots + p^{n-1}\phi(p^m))\cdots\right)\right)^{lp^{n-1}\phi(p^2)}\\
&\quad \frac{\{\Theta \left(K_{\phi(p^{3})}+p^{n-1}\left(K_{\phi(p^{4})}+p^{n-1}(\cdots + p^{n-1}K_{\phi(p^m)})\right)\cdots\right), x-p^2)\}^{lp^{2(n-1)}}}{(x-p^2)^{lp^{n-1}}}
\end{split}
\end{equation*}

Using Theorem \ref{d-union}, Theorem \ref{d-join}, and Lemma \ref{1-phi}, the above polynomial will be

\begin{equation*}
\begin{split}
&=x(x-p^{mn})(x-1)^{l-1}\left(x-p-p^{n-1}\left(\phi(p^{2})+p^{n-1}(\cdots + p^{n-1}\phi(p^m))\cdots\right)\right)^{l\phi(p)}\\
&\quad (x-p)^{l(p^{n-1}-1)} \left(x-p^2-p^{n-1}\left(\phi(p^{3})+p^{n-1}(\cdots + p^{n-1}\phi(p^m))\cdots\right)\right)^{lp^{n-1}\phi(p^2)}\\
&\quad (x-p^2)^{lp^{n-1}(p^{n-1}-1)} \left(x-p^3-p^{n-1}\left(\phi(p^{4})+p^{n-1}(\cdots + p^{n-1}\phi(p^m))\cdots\right)\right)^{lp^{2(n-1)}\phi(p^3)}\\
&\quad \frac{\{\Theta \left(K_{\phi(p^{4})}+p^{n-1}\left(K_{\phi(p^{5})}+p^{n-1}(\cdots + p^{n-1}K_{\phi(p^m)})\right)\cdots\right), x-p^3)\}^{lp^{3(n-1)}}}{(x-p^3)^{lp^{2(n-1)}}}
\end{split}
\end{equation*}
Continuing this process up to $(m-1)^{\mbox{th}}$ step, we get
\begin{equation*}
\begin{split}
&=x(x-p^{mn})(x-1)^{l-1}\left(x-p-p^{n-1}\left(\phi(p^{2})+p^{n-1}(\cdots + p^{n-1}\phi(p^m))\cdots\right)\right)^{l\phi(p)}\\
&\quad (x-p)^{l(p^{n-1}-1)} \left(x-p^2-p^{n-1}\left(\phi(p^{3})+p^{n-1}(\cdots + p^{n-1}\phi(p^m))\cdots\right)\right)^{lp^{n-1}\phi(p^2)}\\
&\quad (x-p^2)^{lp^{n-1}(p^{n-1}-1)} \left(x-p^3-p^{n-1}\left(\phi(p^{4})+p^{n-1}(\cdots + p^{n-1}\phi(p^m))\cdots\right)\right)^{lp^{2(n-1)}\phi(p^3)}\\
&\quad \cdots (x-p^{m-2})^{lp^{(m-1)(n-1)}(p^{n-1}-1)} (x-p^{m-1}-p^{n-1}\phi(p^m))^{lp^{(m-2)(n-1)}\phi(p^{m-1})}\\
&\quad \frac{\{\Theta(K_{\phi(p^{m})}, x-p^{m-1})\}^{lp^{(m-1)(n-1)}}}{(x-p^{m-1})^{lp^{(m-2)(n-1)}}}
\end{split}
\end{equation*}

\begin{equation*}
\begin{split}
&=x(x-p^{mn})(x-1)^{l-1}\left(x-p-p^{n-1}\left(\phi(p^{2})+p^{n-1}(\cdots + p^{n-1}\phi(p^m))\cdots\right)\right)^{l\phi(p)}\\
&\quad (x-p)^{l(p^{n-1}-1)} \left(x-p^2-p^{n-1}\left(\phi(p^{3})+p^{n-1}(\cdots + p^{n-1}\phi(p^m))\cdots\right)\right)^{lp^{n-1}\phi(p^2)}\\
&\quad (x-p^2)^{lp^{n-1}(p^{n-1}-1)} \left(x-p^3-p^{n-1}\left(\phi(p^{4})+p^{n-1}(\cdots + p^{n-1}\phi(p^m))\cdots\right)\right)^{lp^{2(n-1)}\phi(p^3)}\\
&\quad \frac{\{(x-p^{m-1})(x-p^m)^{\phi(p^m)-1}\}^{lp^{(m-1)(n-1)}}}{(x-p^{m-1})^{lp^{(m-2)(n-1)}}}
\end{split}
\end{equation*}
\begin{equation*}
\begin{split}
&= x(x-p^{mn})(x-1)^{l-1} (x-p^m)^{lp^{(m-1)(n-1)}(\phi(p^m)-1)} \left(\prod_{i=1}^{m-1}(x-p^i)^{lp^{(i-1)(n-1)}(p^{n-1}-1)}\right) \\
&\quad \left(\prod_{i=1}^{m-1}\left(x-p^i-p^{n-1}\left(\phi(p^{i+1})+p^{n-1}(\cdots + p^{n-1}\phi(p^m))\cdots\right)\right)^{lp^{(i-1)(n-1)}\phi(p^i)}\right),
\end{split}
\end{equation*}
as desired. This completes the proof.
\end{proof}

For a finite abelian group $G$, the set $\omega(G)$ consists of all positive divisors of $|G|$. Therefore, we have the following straightforward corollary of the Theorem \ref{1-main}.

\begin{corollary}\label{sp-gr-gra}
For integer $n > 1$, we have $\omega(\mathbb{Z}_{p^m}^n) \subseteq \mbox{L-spec}(\mathcal{G}(\mathbb{Z}_{p^m}^n))$.
\end{corollary}

If $n = 1$, the Corollary \ref{sp-gr-gra} is not necessarily true as shown in the following Example \ref{not-spec}.

\begin{example}\label{not-spec}
Consider the finite abelian group $\mathbb{Z}_8$. By Theorem \ref{comp}, the power graph $\mathcal{G}(\mathbb{Z}_8)$ is isomorphic to $K_8$. It is well-known that $\mbox{L-spec}(K_n) = \{0^1, n^{n-1}\}$, therefore we have $\mbox{L-spec}(\mathcal{G}(\mathbb{Z}_8)) = \{0^1, 8^{7}\}$. Since $\mathbb{Z}_8$ is abelian, therefore $\omega(\mathbb{Z}_8) = \{1, 2, 4, 8\}$ and consequently $\omega(\mathbb{Z}_8)\nsubseteq\mbox{L-spec}(\mathcal{G}(\mathbb{Z}_8)$.
\end{example}

The following corollary is a straightforward consequence of the Theorem \ref{1-main}.
\begin{corollary}\

\begin{enumerate}
\item[\rm(i)] The total number of distinct Laplacian eigenvalues of the power graph $\mathcal{G}(\mathbb{Z}_{p^m}^n)$ is $2(m+1)$.
\item[\rm(ii)] The power graph $\mathcal{G}(\mathbb{Z}_{p^m}^n)$ is Laplacian integral.
\end{enumerate}
\end{corollary}

\subsection{Laplacian spectrum of power graph $\mathcal{G}(\mathbb{Z}_2^r \times \mathbb{Z}_{4}^s)$}

In this subsection we determine the Laplacian polynomial of the power graph $\mathcal{G}(\mathbb{Z}_2^r \times \mathbb{Z}_{4}^s)$. Furthermore, for a finite abelian group $H$ whose power graph $\mathcal{G}(H)$ is planar, we observe that the spectrum of group $H$ is contained in the Laplacian spectrum of graph $\mathcal{G}(H)$. Throughout this subsection, we use $G$ to denote the finite abelian group $\mathbb{Z}_2^r \times \mathbb{Z}_{4}^s$. Clearly $|G| = 2^{r+2s}$.
Before computing the Laplacian spectrum of the power graph $\mathcal{G}(G)$, we prove some necessary results.

\begin{theorem}\label{num-2-ord}
The total number of elements of order two and order four in group $G$ are $2^{r+s}-1$ and $2^{r+s}(2^{s}-1)$, respectively.
\end{theorem}

\begin{proof}
By Cauchy's Theorem, group $G$ has at least one element of order two. We first count the number of elements of order two in group $G$.

Let $\alpha = (\alpha_1, \ldots, \alpha_r, \alpha_{r+1}, \ldots, \alpha_s)$ be an arbitrary element of order two in $G$. Then $\alpha_i \in \{0,2\} \subsetneq \mathbb{Z}_4$ for all $i\;(r+1 \le i \le s)$. Since $|\alpha| = 2$, it follows that $\alpha_i \neq 0$ for some $i\;(1\le i \le s)$. Hence the total number of elements of order two in $G$ is $2^{r+s}-1$.

Note that $G$ also has at least one element of order four. Since $|G| = 2^{r+2s}$ and identity is the only element of order one, therefore the total number of elements of order four in $G$ is
\[2^{r+2s} - (2^{r+s}-1) - 1 = 2^{r+s}(2^{s}-1). \]
\end{proof}

\begin{theorem}\label{2-adj-4}
Let $\alpha = (\alpha_1, \ldots, \alpha_r, \alpha_{r+1}, \ldots, \alpha_s)$ be an element of order two in $G$. Then $\alpha \in \langle x\rangle$ for some $x\in G$ of order four if and only if $\alpha_i = 0$ for all $i \;(1\le i \le r)$.
\end{theorem}

\begin{proof}
Let $x = (x_1, \ldots, x_r, x_{r+1},\ldots, x_s)$ be an element of order four in $G$ such that $\alpha \in \langle x\rangle$. Then $2x=\alpha$ and therefore
$2 x_i=\alpha_i$ for all $i \;(1\le i \le s)$. We claim that $\alpha_i = 0$ for all $i\;(1\le i \le r)$. If not, let $\alpha_k = 1$ for some $k\;(1\le k \le r)$. Then $2x_k=1$ which is a contradiction since $x_k \in \mathbb{Z}_2$.

Conversely, we assume that $\alpha_i = 0$ for all $i \;(1\le i \le r)$. Consider an element $x= (x_1, \ldots, x_r, x_{r_1}, \ldots, x_s)$ of order four in $G$ such that $x_i = 0$ for $i\;(1\le i \le r)$, and for $i\;(r+1 \le i \le s)$, if
\begin{enumerate}
\item[\rm(i)] $\alpha_i = 0$, then $x_i = 0$
\item[\rm(ii)] $\alpha_i \neq 0$, then $x_i = 1$.
\end{enumerate}
We claim that $\alpha \in \langle x \rangle$. Since $x_i = 0$ for all $i\;(1 \le i \le r)$, it follows that $2x_i = 0 = \alpha_i$ for $i\;(1 \le i \le r)$.

Now for $i\; (r+1 \le i \le s)$, if $\alpha_i = 0$, then $x_i = 0$ and therefore $2x_i = 0 = \alpha_i$.
Further, if $\alpha_i \neq 0$, then $\alpha_i = 2$. In this case $x_i = 1$ which gives $2 x_i = 2 = \alpha_i$. Consequently $2x = \alpha$, as desired.
This completes the proof.
\end{proof}

The following result characterizes the degree of an element of order two in $G$.

\begin{theorem}\label{deg-2-ord}
Let $\alpha = (\alpha_1, \ldots, \alpha_r, \alpha_{r+1}, \ldots, \alpha_s)$ be a vertex in graph $\mathcal{G}(G)$ such that $|\alpha| = 2$. If
\begin{enumerate}
\item[\rm(i)] $\alpha_i \neq 0$ for some $i\;(1\le i \le r)$, then $\deg(\alpha) = 1$.
\item[\rm(ii)] $\alpha_i = 0$ for all $i\;(1\le i \le r)$, then $\deg(\alpha) = 2^{r+s}+1$.
\end{enumerate}
\end{theorem}

\begin{proof}
Note that the identity element is adjacent to every vertex of  $\mathcal{G}(G)$. Therefore $\deg(\alpha) \ge 1$.
If $\beta$ is another vertex of $\mathcal{G}(G)$ such that $|\beta| = 2$, then clearly $\alpha$ and $\beta$ are non-adjacent in $\mathcal{G}(G)$.

\vspace{0.1cm}

(i) From Theorem \ref{2-adj-4}, it is obvious that $\deg(\alpha) = 1$.

\vspace{0.1cm}

(ii) From Theorem \ref{2-adj-4}, there is an element of order four in $G$ such that the element is adjacent to $\alpha$ in $\mathcal{G}(G)$. Therefore, it is sufficient to count the number of elements of order four in $G$ which are adjacent to $\alpha$.

Let $x = (x_1,\ldots,x_r,x_{r+1},\ldots,x_s)$ be an element of order four in $G$ such that $x$ is adjacent to $\alpha$ in  $\mathcal{G}(G)$. Then $2x =\alpha$ and therefore $2x_i = \alpha_i$ for all $i \;(1\le i \le s)$.

For $i\;(1\le i \le r)$, since $\alpha_i = 0$, it follows that $2x_i = 0$. Therefore, for all $i\;(1 \le i \le r)$, the equations $2x_i= 0$ has $2^r$ solutions.

For $i\;(r+1 \le i \le s)$, since $|\alpha| = 2$, we have $\alpha_i \in \{0, 2\}$. If $\alpha_i = 0$, then $2x_i = 0$ and therefore $x_i \in \{0, 2\}$. If $\alpha_i = 2$, then $2x_i = 2$ and therefore $x_i \in \{1, 3\}$. Therefore, for all $i\;(r+1 \le i \le s)$, the equations $2x_i=\alpha_i$ has $2^s$ solutions.

Thus there are $2^{r+s}$ elements of order four in $G$ which are adjacent to $\alpha$ in power graph $\mathcal{G}(G)$ and consequently $\deg(\alpha) = 2^{r+s}+1$.
\end{proof}

\begin{proposition}\label{deg-2}
The total number of elements of order two in group $G$ having degree $1$ and degree $2^{r+s}+1$ in $\mathcal{G}(G)$ are $2^s(2^r-1)$ and $(2^s-1)$, respectively.
\end{proposition}

\begin{proof}
Let $\alpha = (\alpha_1, \ldots, \alpha_r, \alpha_{r+1}, \ldots, \alpha_s)$ be an element of order two in $G$. If $\deg(\alpha) = 2^{r+s}+1$, then $\alpha_i = 0$ for all $i\;(1\le i \le r)$. Since $|\alpha| = 2$, it follows that $\alpha_i \in \{0, 2\}$ for all $i\;(r+1 \le i \le s)$. Therefore the total number of elements of order two in $G$ having degree $2^{r+s}+1$ in $\mathcal{G}(G)$ is $(2^s-1)$.

From Theorem \ref{num-2-ord}, the total number of elements of order two in $G$ is $2^{r+s}-1$. Therefore the total number of elements of order two in $G$ having degree one in $\mathcal{G}(G)$ is $(2^{r+s}-1) - (2^s-1) = 2^s(2^r-1)$.
\end{proof}

As a direct consequence of Theorem \ref{deg-2-ord} and Proposition \ref{deg-2}, we have the following immediate theorem that describes the power graph $\mathcal{G}(G)$.

\begin{theorem}\label{strt-2}
$\mathcal{G}(G) \cong K_1+\big(2^s(2^r-1)K_1 \cup (2^s-1)(K_1+2^{r+s-1}K_2)\big)$.
\end{theorem}

The next theorem computes the Laplacian spectrum of the power graph $\mathcal{G}(G)$.

\begin{theorem}\label{ls-z24}
The Laplacian polynomial of the power graph $\mathcal{G}(G)$ is
\[x(x-1)^{2^{r+s}-2} (x-2)^{(2^s-1)(2^{r+s-1}-1)} (x-4)^{2^{r+s-1}(2^s-1)}(x-2-2^{r+s})^{2^s-1}(x-2^{r+2s}).\]
\end{theorem}

\begin{proof}
From Theorem \ref{strt-2}, we have \[\mathcal{G}(G) \cong K_1+\big(2^s(2^r-1)K_1 \cup (2^s-1)(K_1+2^{r+s-1}K_2)\big).\]

By using Theorem \ref{d-join}, we get the following Laplacian polynomial $\Theta(\mathcal{G}(G), x)$:
\begin{equation*}
\begin{split}
&  \frac{x(x-2^{r+2s})}{(x-1)(x-2^{r+2s}+1)} \Theta(K_1, x-2^{r+2s}+1) \Theta(2^s(2^r-1)K_1 \cup (2^s-1)(K_1+2^{r+s-1}K_2), x-1)\\
&= \frac{x(x-2^{r+2s})}{(x-1)} \Theta(2^s(2^r-1)K_1 \cup (2^s-1)(K_1+2^{r+s-1}K_2), x-1)\\
&=\frac{x(x-2^{r+2s})}{(x-1)} \left(\Theta(K_1, x-1)\right)^{2^s(2^r-1)} \left(\Theta(K_1+2^{r+s-1}K_2, x-1)\right)^{2^s-1}(\mbox{using Theorem}\; \ref{d-union})
\end{split}
\end{equation*}

Again by using Theorem \ref{d-union} and Theorem \ref{d-join} to the above polynomial, we get
\begin{equation*}
\begin{split}
&= \frac{x(x-2^{r+2s})}{(x-1)} (x-1)^{2^{s}(2^{r}-1)} \left(\frac{(x-1)(x-2-2^{r+s})}{(x-2)(x-1-2^{r+s})} \Theta(K_1,x-1-2^{r+s}) \left(\Theta(K_2, x-2)\right)^{2^{r+s-1}}\right)^{2^s-1}\\
&= \frac{x(x-2^{r+2s})}{(x-1)}(x-1)^{2^{s}(2^{r}-1)}\left(\frac{(x-1)(x-2-2^{r+s})}{(x-2)}\left((x-2)(x-4)\right)^{2^{r+s-1}}\right)^{2^s-1}\\
&= x(x-1)^{2^{r+s}-2} (x-2)^{(2^s-1)(2^{r+s-1}-1)} (x-4)^{2^{r+s-1}(2^s-1)}(x-2-2^{r+s})^{2^s-1}(x-2^{r+2s}).
\end{split}
\end{equation*}
\end{proof}

We know that $\omega(\mathbb{Z}_2^r \times \mathbb{Z}_{4}^s) = \{1,2, 4\}$. Therefore, we have the following straightforward corollary of the Theorem \ref{1-main} and Theorem \ref{ls-z24}.

\begin{corollary} Let $H$ be a finite abelian group such that the power graph $\mathcal{G}(H)$ is planar. Then
\begin{enumerate}
\item[\rm(i)] $\omega(H) \subseteq \mbox{L-spec}(\mathcal{G}(H))$.
\item[\rm(ii)] the graph $\mathcal{G}(H)$ is Laplacian integral.
\end{enumerate}
\end{corollary}

\section{Conclusions}

In the present paper, we determined the Laplacian spectrum of the power graphs $\mathcal{G}(\mathbb{Z}_{p^m}^n)$ and $\mathcal{G}(\mathbb{Z}_2^r \times \mathbb{Z}_{4}^s)$. We proved that the spectrum of group $\mathbb{Z}_{p^m}^n$ is contained in the Laplacian spectrum of graph $\mathcal{G}(\mathbb{Z}_{p^m}^n)$. We also derived a precise formula for the total number of distinct Laplacian eigenvalues of the graph  $\mathcal{G}(\mathbb{Z}_{p^m}^n)$. Finally, for a finite abelian group $G$ whose power graph $\mathcal{G}(G)$ is planar, we proved that the spectrum of group $G$ is contained in the Laplacian spectrum of graph $\mathcal{G}(G)$.

%\bibliographystyle{abbrv}
%\bibliography{shubh-pg}

\begin{thebibliography}{10}


\bibitem{aba13}
 J. ~Abawajy,  A. ~Kelarev, and M. ~Chowdhury.
 \newblock Power graphs: A survey.
\newblock {\em Electronic Journal of Graph Theory and Applications}, 1(2):125--147, 2013.



\bibitem{bapat14}
R. ~B. Bapat.
\newblock{\em Graphs and Matrices}, 2nd edition.
\newblock Springer, 2014.


\bibitem{bapt14}
R. ~B. Bapat and S. ~Roy.
\newblock On the adjacency matrix of a block graph.
\newblock{\em Linear and Multilinear Algebra}, 62(3):406--418, 2014.


\bibitem{came10}
P. ~J. Cameron.
\newblock The power graph of a finite group, II.
\newblock{\em Journal of Group Theory}, 13(6): 779--783, 2010.


\bibitem{came11}
P. ~J. Cameron and S. ~Ghosh.
\newblock The power graph of a finite group.
\newblock{\em Discrete Mathematics}, 311(13): 1220--1222, 2011.



\bibitem{chak09}
I. ~Chakrabarty, S. ~Ghosh, and M. ~K. Sen.
\newblock Undirected power graphs of semigroups.
\newblock {\em Semigroup Forum}, 78(1): 410--426, 2009.


\bibitem{chapan14}
S. ~Chattopadhyay and P. ~Panigrahi.
\newblock Connectivity and planarity of power graphs of finite cyclic, dihedral and dicyclic groups.
\newblock {\em Algebra and Discrete Mathematics}, 18(1): 42--49, 2014.


\bibitem{chapan15}
S. ~Chattopadhyay and P. ~Panigrahi.
\newblock On Laplacian spectrum of power graphs of finite cyclic and dihedral groups.
\newblock {\em Linear and Multilinear Algebra}, 63(7): 1345--1355, 2015.


\bibitem{chat17}
S. ~Chattopadhyay and P. ~Panigrahi.
\newblock On sum of powers of the Laplacian eigenvalues of power graphs of certain finite groups.
\newblock {\em Electronic Notes in Discrete Mathematics}, 63: 137--143, 2017.


\bibitem{chatt17}
S. ~Chattopadhyay, P. ~Panigrahi, and F. ~Atik.
\newblock Spectral radius of power graphs on certain finite groups.
\newblock {\em Indagationes Mathematicae}, In Press.


\bibitem{chel13}
T. ~T. Chelvam  and M. ~Sattanathan.
\newblock Power graph of finite abelian groups.
\newblock {\em Algebra and Discrete Mathematics}, 16(1):33--41, 2013.


\bibitem{dum99}
D.~S. Dummit and R.~M. Foote.
\newblock {\em Abstract Algebra}, 3rd edition.
\newblock John Wiley and Sons, 2004.



\bibitem{kela00}
A. ~V. Kelarev and S. ~J. Quinn.
\newblock A combinatorial property and power graphs of groups.
\newblock {\em Contributions to General Algebra 12 (Vienna, 1999)}, pages 229--235,
\newblock Heyn, Klagenfurt, 2000.



\bibitem{kel65}
A.~K.~Kel'mans.
\newblock The number of trees in a graph {I}.
\newblock {\em Autom. Remote Control}, 26: 2118--2129, 1965.



\bibitem{mck14}
E. ~McKemmie.
\newblock Power graphs of finite groups.
\newblock BA project supervised by Peter M. Neumann at Oxford University, 2014.



\bibitem{mehr17}
Z. ~Mehranian, A. ~Gholami, and A. ~R. ~Ashrafi.
\newblock The Spectra of power graphs of certain finite groups.
\newblock {\em Linear and Multilinear Algebra}, 65(5):1003--1010, 2017.



\bibitem{mohar91}
B. Mohar.
\newblock The {L}aplacian spectrum of graphs. 
\newblock {\em Graph Theory, Combinatorics, and Applications}, pages 871--898, 
\newblock Wiley, New York, 1991.
     



\bibitem{west00}
D. ~B. West.
\newblock {\em Introduction to Graph Theory}, 2nd edition.
\newblock Prentice Hall, 2000.

\end{thebibliography}

\end{document}